\documentclass[12pt]{amsart}
\usepackage{epsfig}
\usepackage{amsmath}
\usepackage{amssymb}
\usepackage{color,soul}

\newcommand{\omitit}[1]{}

\newcommand{\derdir}[2]{{\frac{\partial{#1}}{\partial #2}}}
\newcommand{\derdirtwo}[2]{{\frac{\partial^2{#1}}{\partial{#2}^2}}}
\newcommand{\half}{{\textstyle{1\over 2}}}

\newcommand{\set}[2]{\left\lbrace #1 \; : \; #2 \right\rbrace}

\newcommand{\norm}[1]{\Vert #1 \Vert}

\newcommand{\tbnorm}[1]{\vert\kern-0.1em\vert\kern-0.1em\vert\, #1 \,\vert\kern-0.1em\vert\kern-0.1em\vert}
\newcommand{\qbnorm}[1]{\vert\kern-0.15em\vert\kern-0.15em\vert\kern-0.15em\vert\, #1 \,\vert\kern-0.15em\vert\kern-0.15em\vert\kern-0.15em\vert}

\newcommand{\Omh}{\Omega_h}

\newcommand{\mW}{\mathring{W}_h}
\newcommand{\Bh}{\mathcal{B}_h}
\newcommand{\Th}{\mathcal{T}_h}
\newcommand{\pol}{\mathcal{P}}

\newtheorem{lemma}{Lemma}
\newtheorem{theorem}{Theorem}
\newtheorem{assumption}{Assumption}

\title[Polygonally approximated]{Obtaining  higher-order Galerkin accuracy \\
when the boundary is polygonally approximated}
\date{\today}

\author[Dupont]{Todd Dupont}
\address{The University of Chicago, 
Departments of Computer Science and of Mathematics, Chicago, Illinois, 60637}
\email{dupont@cs.uchicago.edu}

\author[Guzm\'an]{Johnny Guzm\'an}
\address{Division of Applied Mathematics,
Brown University,
Box F,
182 George Street,
Providence, RI 02912}
\email{johnny\_guzman@brown.edu}

\author[Scott]{L.~Ridgway Scott}
\address{The University of Chicago, Emeritus, Chicago, Illinois, 60637}
\email{ridg@uchicago.edu}

\begin{document}

\maketitle

\begin{abstract}

We study two techniques for correcting the geometrical error associated
with domain approximation by a polygon.
The first was introduced some time ago \cite{bramble1972projection} and
leads to a nonsymmetric formulation for Poisson's equation.
We introduce a new technique that yields a symmetric formulation and
has similar performance.
We compare both methods on a simple test problem.
\end{abstract}

\section{Introduction}

When a Dirichlet problem on a smooth domain is approximated by a polygon, 
an error occurs that is suboptimal for quadratic 
approximation \cite{lrsBIBaa,lrsBIBab,lrsBIBae}.
However, this can be corrected by a modification of the variational form
\cite{bramble1972projection}.
Here we review this approach and suggest a new one.

Let $\Omega$ be a smooth, bounded, two-dimensional domain.
Consider the Poisson equation with
Dirichlet boundary conditions: 
\begin{equation}\label{eqn:simplpder}
-\Delta u=f\hbox{ in }\Omega,\quad u=g\hbox{ on }\partial\Omega.
\end{equation}
We assume that $f$ and $g$ are sufficiently smooth that $u$ can be extended to be 
in $H^{k+1}(\widehat\Omega)$, where $\widehat\Omega$ contains a neighborhood 
of the closure of $\Omega$.

One way to discretize \eqref{eqn:simplpder} is to approximate the domain
$\Omega$ by polygons $\Omega_h$, where the edge lengths of $\partial\Omega_h$
are of order $h$ in size.
Then conventional finite elements can be employed, with the Dirichlet boundary 
conditions being approximated by the assumption that $u_h=\hat g$ on 
$\partial\Omega_h$ \cite{lrsBIBgd}, with $\hat g$ appropriately defined.
For example, let us suppose for the moment that $g\equiv 0$ and we take 
$\hat g\equiv 0$ as well.
In particular, we assume that $\Omega_h$ is triangulated with a quasi-uniform
mesh $\Th$ of maximum triangle size $h$, and the boundary vertices of $\Omega_h$ are in $\partial \Omega$. We define $\mW^k:= H_0^1(\Omega) \cap W_h^k$  where 
\begin{equation*}
W_h^k=\{ v \in C(\Omh): v|_T \in \pol_k(T), \forall T \in \Th\}.
\end{equation*}
Then the standard finite element approximation finds $u_h\in \mW^k$ satisfying 
\begin{equation}\label{eqn:polyapprocirkl}
a_h(u_h,v)=(f,v)_{L^2(\Omh)},\quad \forall v\in \mW^k,
\end{equation}
where $a_h(u,v):=\int_{\Omh} \nabla u\cdot\nabla v\,dx$. Here we assume that $f$ is extended smoothly outside of $\Omega$.

This approach for $k=1$ (piecewise linear approximation) leads to the error 
estimate
$$
\norm{u-u_h}_{H^1(\Omega_h)}\leq C h \norm{u}_{H^2(\hat{\Omega})}.
$$
However, when this approach is applied with piecewise quadratic 
polynomials ($k=2$), the best possible error estimate is
\begin{equation}\label{eqn:bestpossib}
\norm{u-u_h}_{H^1(\Omega_h)}\leq C h^{3/2} ,
\end{equation}
which is less than optimal order by a factor of $\sqrt{h}$.
The reason of course is that we have made only a piecewise linear approximation
of $\partial\Omega$.
Table \ref{tabl:jpcircle} summarizes some computational experiments for the
test problem in Section \ref{circle}.
We see a significant improvement for quadratics over linears, but there 
is almost no improvement with cubics.
Moreover, we will see that a significant improvement using quadratics can 
be obtained using simple approaches that modify the variational form.

There have been many techniques introduced to circumvent the loss of accuracy
with quadratics (and higher-order piecewise polynomials)
\cite{lrsBIBae,ref:stenbergNitscheMethod}.
However, all of them require some modification of the quadrature for the elements
at the boundary.

Here we review an approach by Bramble et al. \cite{bramble1972projection} that solves directly 
on $\Omega_h$, but with a modified variational form based on the method of 
Nitsche \cite{ref:stenbergNitscheMethod}.
The method \cite{bramble1972projection} has been modified and applied in
many ways \cite{ref:CutFEMbasedonBDT}. However, the method in  \cite{bramble1972projection} leads to a non symmetric bilinear form.  Given this shortcoming we define a new method that is symmetric and solves the problem on $\Omh$ that has similar convergence results.  As we will see in the next section, one main idea in \cite{bramble1972projection} is that one uses a Taylor series of the solution near the boundary to define appropriate boundary conditions on $\partial \Omh$.  We should mention that this idea has been used recently (see for example \cite{ Cockburn2012, main}).

\begin{table}
\begin{center}
\begin{tabular}{|c|c||c|c||c|c||c|c|c|}\hline
$k$ &$M$&  L2 err&rate& H1 err&rate&seg&hmax\\
\hline
  1&  2&1.84e+00& NA&6.25e+00 & NA & 10&  1.05e+00 \\
  1&  4&2.93e-01&2.65 & 1.89e+00 & 1.73 & 20& 4.94e-01 \\
  1&  8&9.55e-02& 1.62 &1.06e+00 & 0.83 & 40& 2.61e-01 \\
  1& 16&2.47e-02& 1.95 &5.45e-01 & 0.96&  80& 1.35e-01 \\
\hline
  2&  2&4.18e-01& NA & 1.41e+00 & NA & 10& 1.05e+00 \\
  2&  4&9.44e-02& 2.15 &4.26e-01 & 1.73  &20& 4.94e-01 \\
  2&  8&2.30e-02& 2.04 &1.59e-01 & 1.42 & 40& 2.61e-01 \\
  2& 16&5.62e-03& 2.03 &5.45e-02 & 1.54 &  80& 1.35e-01 \\
\hline
  3&  2&3.17e-01&  NA &8.25e-01 & NA  & 10& 1.05e+00 \\
  3&  4&8.81e-02& 1.85  &2.94e-01 & 1.49 &20& 4.94e-01 \\
  3&  8&2.22e-02& 1.99 &1.07e-01 & 1.46& 40& 2.61e-01 \\
  3& 16&5.53e-03& 2.01 &3.82e-02 &1.49&  80& 1.35e-01 \\
\hline
\end{tabular}
\end{center}
\caption{Errors $u_h-u_I$ in $L^2(\Omega_h)$ and $H^1(\Omega_h)$,
as a function of the maximum mesh size (hmax) for the polygonal 
approximation \eqref{eqn:polyapprocirkl} for test problem 
in Section \ref{circle} using various polynomial degrees $k$.
Key: ``$M$'' is input parameter to {\tt mshr} function {\tt circle} used 
to generate the mesh, ``seg'' is the number of boundary edges.
The approximate solutions were generated using \eqref{eqn:polyapprocirkl}.}
\label{tabl:jpcircle}
\end{table}

\section{The Bramble-Dupont-Thom\'ee  approach}
\label{sec:BDTmeth}

\begin{figure}
\centerline{(a)\includegraphics[width=2.5in]{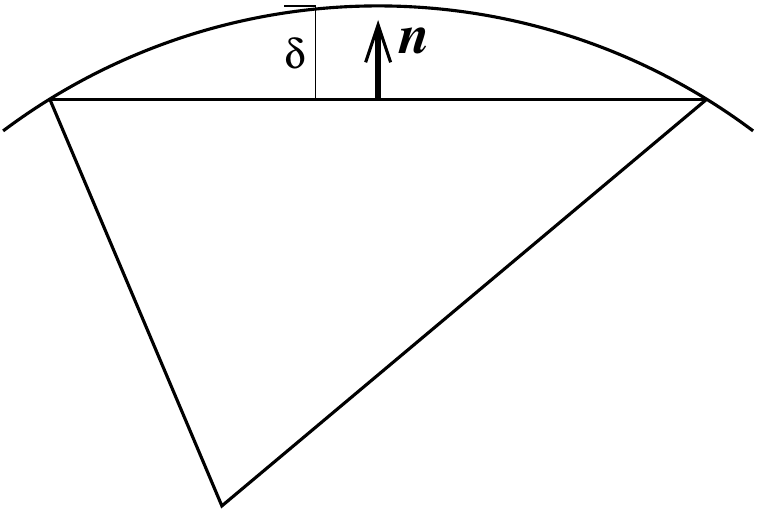}
\qquad   (b)     \includegraphics[width=2.5in]{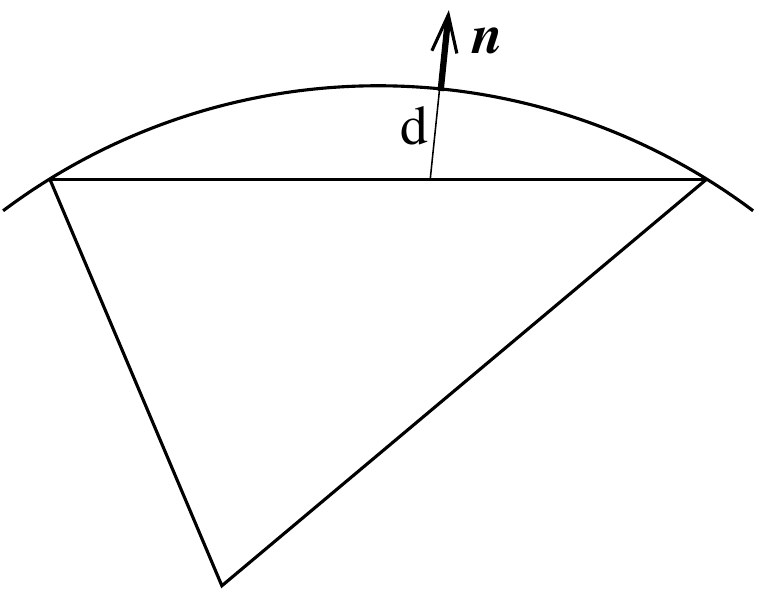}}
\caption{Definitions of (a) $\delta$ and (b) $d$.} 
\label{fig:bdryedge}
\end{figure}

The method \cite{bramble1972projection} of Bramble-Dupont-Thom{\' e}e (BDT)
achieves high-order accuracy by modifying Nitsche's 
method \cite{ref:stenbergNitscheMethod} applied on $\Omega_h$. We assume that $\Omega_h \subset \Omega$ and we do not necessarily assume that the boundary vertices of $\Omega_h$ belong to $\partial \Omega$. The bilinear form used in  \cite{bramble1972projection} is
\begin{equation}\label{eqn:toddformr}
N_h(u,v)=a_h(u,v)-\int_{\partial\Omega_h} \derdir{u}{n} v \,ds
-\int_{\partial\Omega_h} 
\Big(u+\delta\derdir{u}{n}\Big)\Big(\derdir{v}{n} -\gamma h^{-1} v\Big) \,ds
\end{equation}
Here, $n$ denotes the outward-directed normal to $\partial\Omega_h$ and
$$
\delta(x)=\min\set{s>0}{x+sn\in\partial\Omega}.
$$

Contrast the definition of $\delta$ to the closely related function $d$ defined by
$$
d(x)=\min\set{|x-y|}{y\in\partial\Omega}.
$$

For simplicity the assume that $g=0$. Then the BDT method will find $u_h \in W_h^k$ such that 
 \begin{equation*}
 N_h(u_h,v)=\int_{\Omega_h} fv\,dx \qquad \text{ for all } v \in W^k_h. 
\end{equation*}
If $\delta$ were 0, this would be Nitsche's method on $\Omega_h$.

Corrections of arbitrary order, involving terms 
$\delta^\ell\, {\frac{\partial^\ell u}{\partial n^\ell}}$ for $\ell>1$
are studied in \cite{bramble1972projection}, but for simplicity we restrict
attention to the first-order correction to Nitsche's method given in \eqref{eqn:toddformr}. The error estimates obtained in  \cite{bramble1972projection} are as follows

$$
\tbnorm{u-u_h}_1
\leq Ch^k\norm{u}_{H^{k+1}(\Omega)}+ Ch^{7/2}\norm{u}_{W^{2}_\infty(\Omega)},
$$
where
$$
\tbnorm{v}_1:=\Big(a_h(v,v) +h^{-1}\int_{\partial\Omega_h}v^2\,ds
+h\int_{\partial\Omega_h}\Big(\derdir{v}{n}\Big)^2\,ds\Big)^{1/2}.
$$
Thus using the variational form \eqref{eqn:toddformr} leads to an approximation 
that is optimal-order with quadratics and cubics and is only suboptimal for
quartics by a factor of $\sqrt{h}$.

\begin{figure}
\centerline{(a)\includegraphics[width=3.0in]{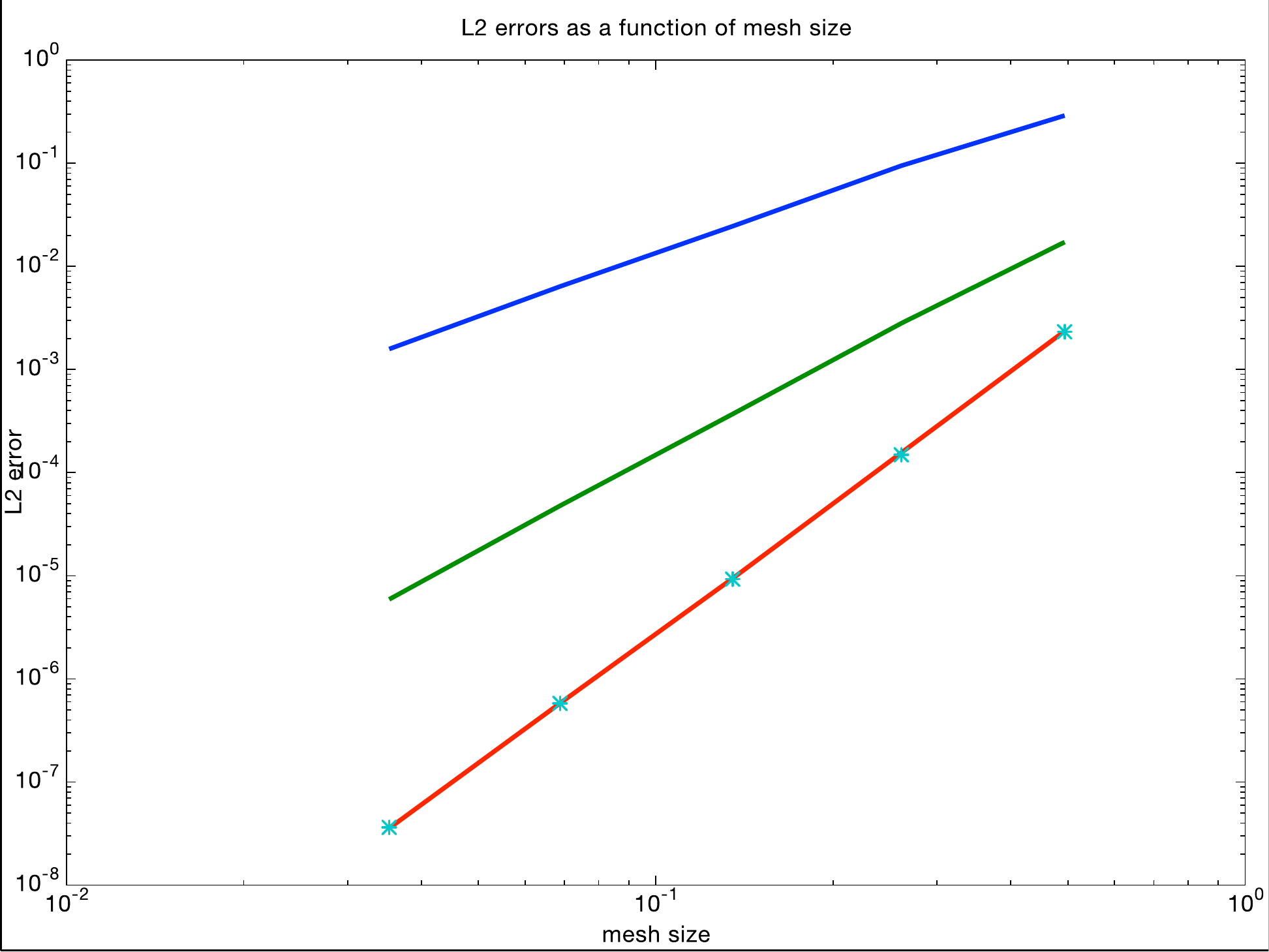}
(b)\includegraphics[width=3.0in]{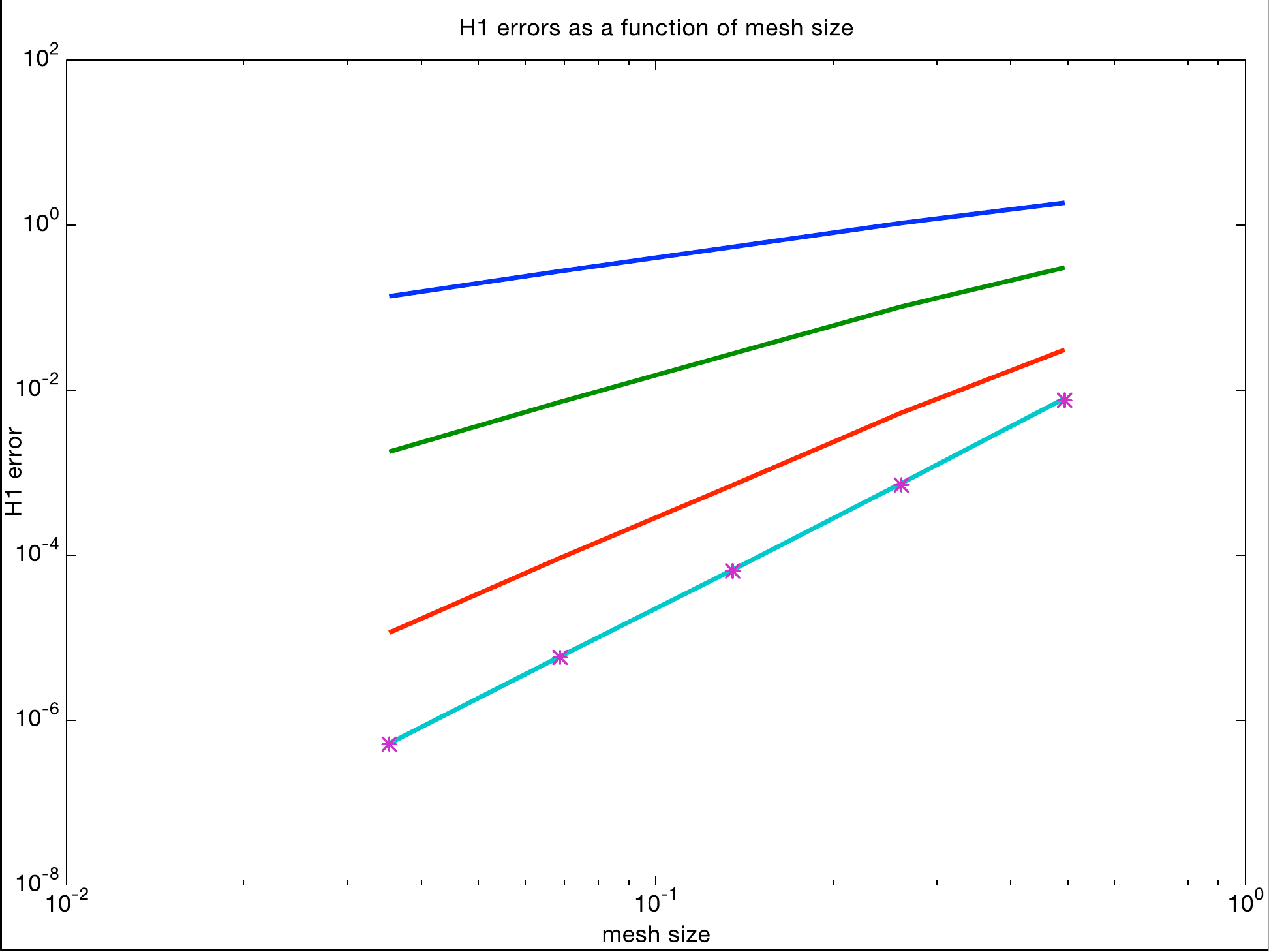}}
\caption{Errors $u_h-u_I$ in (a) $L^2(\Omega_h)$ and (b) $H^1(\Omega_h)$ as 
a function of the maximum mesh size for the BDT method with $\gamma=100$.
The asterisks indicate data for (a) $k=4$ and (b) $k=5$.}
\label{fig:plotshortbdt}
\end{figure}

\subsection{An example of a circle}\label{circle}
We consider a numerical example.  Consider the case where $\Omega$ is a disc of radius $R$ centered at the origin,
in which case we have $d(x)=R-|x|$.
However, it is more difficult to evaluate $\delta(x)$.
We have $x+\delta(x)n\in\partial\Omega$ for $x\in\partial\Omega_h$, 
where $n$ denotes the outward normal to $\Omega_h$.
We can write $x=(x\cdot n)\,n+(x\cdot t)\,t$, and
$(x\cdot t)^2=|x|^2-(x\cdot n)^2$.
Since $|x+\delta(x)n|=R$, we have
$$
R^2=(x\cdot t)^2+(x \cdot n+\delta(x))^2
 =|x|^2-(x \cdot n)^2+((x\cdot n+\delta(x))^2.
$$
Then 
$$
\delta(x)=\pm\sqrt{R^2-|x|^2+(x\cdot n)^2}-x\cdot n \, .
$$
Note that for $x\in\partial\Omega_h$, $|x|\leq R$ and $x\cdot n>0$.
Since $\delta(x)\geq 0$, we must pick the plus sign, so
$$
\delta(x)=\sqrt{R^2-|x|^2+(x\cdot n)^2}-x\cdot n \, .
$$
It is not hard to see that $d-\delta=\mathcal{O}(h^4)$ in this case.

This problem is simple to implement using the FEniCS Project code
{\tt dolfin} \cite{fenicsbook}.
We take $R=1$, $u(x,y)=1-(x^2+y^2)^3$, and $f=36(x^2+y^2)^2$ 
in the computational experiments described subsequently.
Computational results for this example are given in Table \ref{tabl:rateshortbdt}
where we see optimal order approximation for $k\leq 3$, improvement for
$k=4$ over $k=3$ (suboptimal by a factor $h^{-1/2}$), and no improvement
for quintics. These errors are depicted in Figure \ref{fig:plotshortbdt}.

\begin{table}
\begin{center}
\begin{tabular}{|c|c|c||c|c||c|c|}\hline
$k$&$M$ &hmax&  L2 error&rate& H1 error&rate\\
\hline
 1   &  8 &   0.261 &  0.0947 &  1.61 &    1.06 & 0.82 \\
 1   & 16 &   0.135 &  0.0245 &  1.95 &   0.544 & 0.96 \\
 1   & 32 &  0.0688 & 0.00639 &  1.94 &   0.277 & 0.97 \\
 1   & 64 &  0.0353 & 0.00158 &  2.02 &   0.137 &  1.02 \\
\hline
 2   &  8 &   0.261 & 2.81e-03 &  2.61 &   0.103 &  1.57 \\
 2   & 16 &   0.135 & 3.70e-04&   2.93 &  0.0277 &  1.89 \\
 2   & 32 &  0.0688 &  4.77e-05 &  2.96 & 0.00717 &  1.95 \\
 2   & 64 &  0.0353 &  5.91e-06 &  3.01 & 0.00179 &   2.00 \\
\hline
 3   &  8 &   0.261 & 1.56e-04 &  3.92 & 5.31e-03 &  2.54 \\
 3   & 16 &   0.135 &  9.44e-06 &  4.05 &  7.06e-04&   2.91 \\
 3   & 32 &  0.0688 &  5.81e-07 &  4.02 &  9.23e-05 &  2.94 \\
 3   & 64 &  0.0353 &  3.57e-08 &  4.02 &  1.15e-05 &  3.00 \\
\hline
 4   &  8 &   0.261 &  1.49e-04 &  3.96 &  7.41e-04 &  3.42 \\
 4   & 16 &   0.135 &  9.29e-06 &  4.00 &  6.63e-05 &  3.48 \\
 4   & 32 &  0.0688 & 5.80e-07 &  4.00 & 5.90e-06 &  3.49 \\
 4   & 64 &  0.0353 &  3.63e-08 &   4.00 &  5.22e-07 &  3.50 \\
\hline
 5   &  8 &   0.261 &  1.47e-04 &  3.96 & 7.10e-04&  3.41 \\
 5   & 16 &   0.135 &  9.27e-06 &  3.99 &  6.44e-05 &  3.46 \\
 5   & 32 &  0.0688 & 5.80e-07 &  4.00   &5.77e-06 &  3.48 \\
 5   & 64 &  0.0353 &  3.62e-08 &   4.00 &  5.12e-07 &  3.49 \\
\hline
\end{tabular}
\end{center}
\caption{Errors $u_h-u_I$ in $L^2(\Omega_h)$ and $H^1(\Omega_h)$ as a function of
mesh size (hmax) for the the BDT approximation in Section \ref{sec:BDTmeth},
with $\gamma=100$, for various polynomial degrees $k$.
Key: $M$ is the value of the {\tt meshsize} input parameter to the {\tt mshr}
function {\tt circle} used to generate the mesh.
The number of boundary edges was set to $5M$, and
hmax is the maximum mesh size.}
\label{tabl:rateshortbdt}
\end{table}

\section{A new method based on a Robin-type approach}

One issue with the BDT method is that the resulting linear system is not symmetric,
although it is possible to symmetrize the method as we discuss in Section \ref{sec:hoasm}.
Here we develop a technique that leads to a symmetric system.
Moreover, this method does not require the parameter(s) from Nitsche's method.
For Nitsche's method to succeed, $\gamma$ must be chosen appropriately 
\cite{lrsBIBih}.

We first separate $\partial \Omega$ to its piecewise linear part and its curvilinear part. We will assume that $\partial \Omega=\Gamma^0 \cup S_1 \cup \ldots S_\ell$ where $\Gamma^0 $ is a piecewise linear segment and $S_i's$ are $C^2$ and no where linear. We let the end points of $S_i$ to be $y_{i-1}, y_{i}$.  

For the method in this section we assume that the vertices of $\Omh$ belong to $\partial \Omega$ and hence $\Omh$ might not be a subdomain of $\Omega$. Hence, we need to define $\delta$ in this case.  We assume that for every $x \in \partial \Omh$ that is  there is a unique smallest number $\delta(x)$ in absolute value such that $\Omega \backslash \Gamma^0 $
\begin{equation*}
x+\delta(x) n(x) \in \partial \Omega.
\end{equation*}

We assume that the approximate domain boundary $\partial\Omega_h$
can be decomposed into three parts, as follows. Let $\mathcal{E}_h$ be the edges of $\partial \Omega_h$.
\begin{equation}\label{eqn:domaindec}
\Gamma^{\pm}= \cup\set{e \in \mathcal{E}_h }{\pm\delta|_{e^o}>0},
\end{equation}
where $e^o$ denotes the interior of $e$. Let $\Gamma=\Gamma^{+}\cup\Gamma^{-}$.  We assume the following.
\begin{assumption}\label{assum2}
We assume that all the vertices of $\partial \Omh$ belong to $\partial \Omega$ and that  each $y_i$ (for $0 \le i \le \ell$) is a vertex of $\partial \Omh$ . Finally, we assume that 
\begin{equation*}
\partial  \Omh=\Gamma^0 \cup \Gamma.
\end{equation*}
\end{assumption}

Our method is based on a Robin type of boundary condition on $\Gamma$. In fact, our method will be based on the closely related   problem:
 \begin{alignat*}{2}
-\Delta w=&f, \quad && \text{ on  } \Omega, \\
w=&0, \quad && \text{ on } \Gamma^0, \\
w+\delta \frac{\partial w}{\partial n}=&\hat{g}, \quad  && \text{ on } \Gamma. 
\end{alignat*}
Here we define $\hat{g}(x)= g(x+\delta(x) n(x))$ for $x \in \Gamma$ and not a vertex of $\partial \Omh$.  The key here is that, using that $u$ vanishes on $\partial \Omega$, for $x \in \Gamma$ ($x$ not a vertex of  $\partial \Omh$) we have 
\begin{equation}\label{taylor}
u(x)+\delta \frac{\partial u}{\partial n}(x)=\hat{g}(x)-\frac{\delta^2}{2} \partial_{nn} u(z),
\end{equation}
where $z$ lies in the line segment with end points $x$ and $x+\delta(x) n(x)$. 

Now we can write the method.  We start by defining the finite element space we will use
\begin{equation*}
V_h^k=\{W_h^k: v=0 \text{ on } \Gamma^0, v(x)=0 \text{ for all vertices of } x \text{ of  } \partial \Omega_h\}.
\end{equation*}
Also define 
\begin{equation*}
V_h^k(g)=\{W_h^k: v=g_I \text{ on } \Gamma^0, v(x)=Ig(x) \text{ for all vertices of } x \text{ of  } \partial \Omega_h\}.
\end{equation*}
where $g_I \in C(\partial \Omh)$ is a suitable approximation of $g$ and is a piecewise polynomial of degree  at most $k$ on $\partial \Omh$.

The bilinear form is given by
\begin{equation*}
b_h(u,v):=a_h(u,v)+c_h(u,v),
\end{equation*}
where 
\begin{equation*}
c_h(u,v)=\int_{\Gamma}\delta^{-1}{u}v \,ds.
\end{equation*}

Then the method solves: 

Find $u_h \in V_h^k(g)$ such that  
\begin{equation}\label{fem}
b_h(u_h, v)= \int_{\Omega_h} F v+ \int_{\Gamma}\delta^{-1}{\hat{g}}v \,ds. \, \quad \text{ for all } v \in V_h^k.
\end{equation}
Here
$$
F=\begin{cases} f &\hbox{on}\;\Omega \cap \Omega_h \\
I^1f &\hbox{on}\;\Omega_h \backslash \Omega,\end{cases}
$$
where $I^1$ is the linear interpolant onto $W_h^1$. Note that we can define $I^1 f$ only knowing $f$ on $\Omega$. Alternatively, if we have an analytic representation of $f$ we can define $F$ as a smooth extension of $f$ outside of $\Omega$.

\section{Error Analysis}
\label{sec:newproof}

\subsection{Stability Analysis}
Unfortunately, the bilinear form $b_h$ is not positive definite. However, we will  be able to prove stability of  method. In order to do so, we need to decompose the space $V_h^k$ into its boundary contribution and interior contribution. More precisely, we can write
\begin{equation*}
V_h^k=\mW^k \oplus \Bh^k,
\end{equation*}
where $\Bh^k =\{ v\in V_h^k: v(x)=0 \text{ for all interior Lagrange points } x \}$.  
We will define a norm on  $V_h^k$:
\begin{equation*}
\|v\|_a^2:=a_h(v,v)  
\end{equation*}
and a semi-norm 
\begin{equation*}
|v|_c^2:=\int_{\Gamma} \frac{v^2}{|\delta|}  \, ds.
\end{equation*}
Note that $|\cdot |_c$ is in fact a  norm on $\Bh^k$.

The following crucial lemma will allow us to prove stability.  
\begin{lemma}
\label{lem:lemseven}
There exists a constant $c_1>0$ such that 
\begin{equation}\label{lemma1}
\|v\|_{a} \le c_1 \sqrt{h} |v|_c \text{ for all }  v\in \Bh^k.
\end{equation}
\end{lemma}
\begin{proof}
Let $\mathcal{E}_h^\Gamma$ be the collection of edges that are a subset of $\Gamma$ and let $\mathcal{T}_h^\Gamma$ be triangles $T$ such that $T$ has an edge in $\mathcal{E}_h^\Gamma$. Then, if $v\in \Bh^k$ and using inverse estimates we have
\begin{equation*}
\|v\|_{a}^2=\sum_{T \in \mathcal{T}_h^\Gamma} \|\nabla v\|_{L^2(T)}^2 
\le \sum_{T \in \mathcal{T}_h^\Gamma} \frac{C}{h_T^2} \|v\|_{L^2(T)}^2 
\le \sum_{e \in \mathcal{E}_h^\Gamma} \frac{C}{h_e} \|v\|_{L^2(e)}^2.
\end{equation*}
The result is complete after we use that $\max_{x \in e} |\delta(x)| \le C h_e^2$ for $e \in \mathcal{E}_h^\Gamma$.
\end{proof}

We note that $c_h(u,v)$ may not be well defined for all $u, v \in V_h^k$. Therefore, we need to make an assumption on $\delta$ such that  this is not the case. 

\begin{assumption}\label{assumption3}
We assume that $\delta$ is such that 
\begin{equation}\label{aux321}
|c_h(u,v)| < \infty \qquad \forall u, v\in V_h^k.
\end{equation}
\end{assumption}

For example, if $\delta$ has a lower bound as follows, then  \eqref{aux321} will hold.
Suppose that the end points of $e \in \mathcal{E}_h^\Gamma$ are $x_0$ and $x_1$.
Then we assume that there exists a constant $c>0$ and a $p<3$ such that
\begin{equation*}
 |x-x_0|^p |x-x_1|^p \le c  |\delta(x)| \quad \text{ for all } x\in e,
\end{equation*}
where $c$ is independent of $e \in \mathcal{E}_h^\Gamma$. 
Under these conditions, Assumption \ref{assumption3} holds.

We can now prove the stability result. 
\begin{theorem}\label{stability}
We assume that Assumption \ref{assum2}  and Assumption \ref{assumption3} hold.  Suppose that $G$ is a bounded linear function on $V_h^k$ and suppose that $u_h \in V_h^k$ solves
\begin{equation*}
b_h(u_h, v)= G(v), \quad  \text{ for all } v \in V_h^k. 
\end{equation*}
Then, assuming $c_1 \sqrt{h} \le \frac{1}{2}$ we have 
\begin{equation*}
\|u_h\|_a \le   2 \left(\sup_{v_h \in \mW^k} \frac{|G(v_h)|}{\|v_h\|_a}\right)+ \frac{11}{3} c_1 \sqrt{h} \left(\sup_{v_h \in \Bh^k} \frac{|G(v_h)|}{|v_h|_c}\right).
\end{equation*}
and
\begin{equation*}
|u_h|_c \le  \frac{3}{2}\left(\sup_{v_h \in \mW^k} \frac{|G(v_h)|}{\|v_h\|_a}\right) + \frac{5}{3}  \left(\sup_{v_h \in \Bh^k} \frac{|G(v_h)|}{|v_h|_c}\right).
\end{equation*}
\end{theorem}
\begin{proof}
We know we can write $u_h =w_h+ s_h$ where $w_h \in \mW^k$ and $s_h \in \Bh^k$.  Define $\phi_h\in \Bh^k$ by
$$
\phi_h=\begin{cases} s_h &\hbox{on}\;\Gamma^{+}\\
-s_h &\hbox{on}\;\Gamma^{-}\\
            0 &\hbox{on}\;\Gamma^0 .\end{cases}
$$
Note that $|\phi_h|_c=|s_h|_c$. Now we can estimate $s_h$.

\begin{alignat*}{1}
|s_h|_c^2=c_h(s_h, \phi_h)=b_h(u_h, \phi_h)-a_h(u_h, \phi_h)=G(\phi_h)-a_h(u_h, \phi_h). 
\end{alignat*}
Hence, we have 
\begin{alignat*}{1}
|s_h|_c^2 \le & \left(\sup_{v_h \in \Bh^k} \frac{|G(v_h)|}{|v_h|_c}\right) |\phi_h|_c+ \|u_h\|_a \|\phi_h\|_a  \\
 \le & \left(\sup_{v_h \in \Bh^k} \frac{|G(v_h)|}{|v_h|_c}\right) |s_h|_c
        + c_1 \sqrt{h} (\|w_h\|_a + c_1 \sqrt{h} |s_h|_c) |s_h|_c.
 \end{alignat*}
 
 Here we used \eqref{lemma1} twice. In particular, we used $\|u_h\|_a \le \|w_h\|_a+\|s_h\|_a \le \|w_h\|_a + c_1 \sqrt{h} |s_h|_c$. Assuming $h  c_1^2 \le \frac{1}{4}$  we have
\begin{alignat*}{1}
\frac{3}{4}|s_h|_c^2 \le \left(\sup_{v_h \in \Bh^k} \frac{|G(v_h)|}{|v_h|_c}\right) |s_h|_c
  + \frac{1}{2} \|w_h\|_a  |s_h|_c.
\end{alignat*}
Hence, 
\begin{equation}\label{341}
|s_h|_c \le \frac{4}{3}\left(\sup_{v_h \in \Bh^k} \frac{|G(v_h)|}{|v_h|_c}\right) 
+ \frac{2}{3} \|w_h\|_a
\end{equation}
Next, 
\begin{equation*}
\|w_h\|_a^2= a_h(w_h, w_h)=a_h(u_h, w_h)-a_h(s_h, w_h)=b_h(u_h,w_h)-a_h(s_h, w_h)=G(w_h)-a_h(s_h, w_h).
\end{equation*}
We therefore have 
\begin{equation*}
\|w_h\|_a^2 \le  \left(\sup_{v_h \in \mW^k} \frac{|G(v_h)|}{|v_h|_a}\right) \|w_h\|_a+ \|s_h\|_a \|w_h\|_a.
\end{equation*}
Hence, we obtain using \eqref{341}
\begin{alignat*}{1}
\|w_h\|_a \le & \left(\sup_{v_h \in \mW^k} \frac{|G(v_h)|}{\|v_h\|_a}\right) + \|s_h\|_a   \\
& \le  \left(\sup_{v_h \in \mW^k} \frac{|G(v_h)|}{\|v_h\|_a}\right) + c_1 \sqrt{h}\|s_h\|_c \\
& \le  \left(\sup_{v_h \in \mW^k} \frac{|G(v_h)|}{\|v_h\|_a}\right) +  \frac{4}{3} c_1 \sqrt{h} \left(\sup_{v_h \in \Bh^k} \frac{|G(v_h)|}{|v_h|_c}\right) + \frac{1}{3} \|w_h\|_a
\end{alignat*}
Thus we arrive at
\begin{equation*}
\|w_h\|_a \le   \frac{3}{2} \left(\sup_{v_h \in \mW^k} \frac{|G(v_h)|}{\|v_h\|_a}\right)+ 2 c_1 \sqrt{h} \left(\sup_{v_h \in \Bh^k} \frac{|G(v_h)|}{|v_h|_c}\right).
\end{equation*}

From this  and \eqref{341} we get
\begin{equation*} 
|u_h|_c=|s_h|_c \le  \frac{3}{2}\left(\sup_{v_h \in \mW^k} \frac{|G(v_h)|}{\|v_h\|_a}\right) + \frac{5}{3}  \left(\sup_{v_h \in \Bh^k} \frac{|G(v_h)|}{|v_h|_c}\right).
\end{equation*}
Finally, 
\begin{alignat*}{1}
\|u_h\|_a \le & \|w_h\|_a + \|s_h\|_a \le  \|w_h\|_a + c_1 \sqrt{h}\|s_h\|_c  \\
          \le & 2 \left(\sup_{v_h \in \mW^k} \frac{|G(v_h)|}{\|v_h\|_a}\right)
        + \frac{11}{3} c_1 \sqrt{h} \left(\sup_{v_h \in \Bh^k} \frac{|G(v_h)|}{|v_h|_c}\right).
\end{alignat*}
\end{proof}

We can now prove error estimates after we make an assumption more stringent 
than Assumption \ref{assumption3}. 
\begin{assumption}\label{assumption1}
 Suppose that the end points of $e \in \mathcal{E}_h^\Gamma$ are $x_0$ and $x_1$.
Then we assume that there exists a constant $\beta>0$ such that
\begin{equation*}
 |x-x_0| |x-x_1| \le \beta |\delta(x)| \quad \text{ for all } x \in e,
\end{equation*}
where $\beta$ is independent of $e \in \mathcal{E}_h^\Gamma$. 
\end{assumption}
Note that this assumption does not allow $\partial \Omega$ and $\partial \Omh $ to be tangent on the vertices of $\Gamma$.
Assumption \ref{assumption1} implies Assumption \ref{assumption3}; 
in particular, the example after Assumption \ref{assumption3} holds with $p=1$.

\begin{theorem}\label{errorestimates}
We assume Assumptions \ref{assum2} and \ref{assumption1} hold.  We assume that $u$ solves \eqref{eqn:simplpder}  and belongs to $u \in W^{s,\infty}(\Omega)$ where $s =\max \{k+1, 4\}$. We assume that $g_I= u_I|_{\partial \Omh}$ where $u_I \in W_h^k$ is the Lagrange interpolant of $u$.  Let $u_h \in V_h^k(g)$ solve \eqref{fem} and assume that $u$ solves \eqref{eqn:simplpder} then we have
\begin{alignat*}{1}
\|u-u_h\|_{a} \le & C h^{k}\|u\|_{H^{k+1}(\hat{\Omega})} + C h^{k+1/2} \|u\|_{W^{k+1,\infty}(\Gamma)} \\
&+C\left(h^4  \|u\|_{W^{4,\infty}(\hat{\Omega})}+ h^{7/2}  \|u\|_{W^{2,\infty}(\hat{\Omega})}\right).
\end{alignat*}
and
\begin{alignat*}{1}
|u-u_h|_{c} \le & C h^{k}\|u\|_{H^{k+1}(\hat{\Omega})} + C h^{k} \|u\|_{W^{k+1,\infty}(\Gamma)} \\
&+C\left(h^4  \|u\|_{W^{4,\infty}(\hat{\Omega})}+ h^{3}  \|u\|_{W^{2,\infty}(\hat{\Omega})}\right).
\end{alignat*}

\end{theorem}

\begin{proof}
 We let $e_h=u_I-u_h \in V_h^k$. Then we see that
\begin{equation*}
b_h(e_h, v)=G(v) \quad \text{ for all } v \in V_h^k,
\end{equation*}
where $G(v)=G_1(v)+ G_2(v)$,  $G_1(v)=\int_{\Omh} F v dx -b_h(u, v)$ and $G_2(v)= b_h(u-u_I, v)$.

Note that  using integration by parts we have
\begin{alignat*}{1}
G_1(v)=& \int_{\Omh} F v+ \int_{\Gamma}\delta^{-1}{\hat{g}}v \,ds-\int_{\Omh} (-\Delta u) v dx - \int_{\Gamma} \left(\frac{\partial u}{\partial n} + \frac{1}{\delta} (u-\hat{g})\right) v \\
=& \int_{\Omh \backslash \Omega} (I^1 (-\Delta u)-(-\Delta u) )v dx  - \int_{\Gamma} \left(\frac{\partial u}{\partial n} + \frac{1}{\delta} (u-\hat{g})\right) v .
\end{alignat*}  

First consider $v \in \mW^k$ then we have
\begin{alignat*}{1}
|G_1(v)| \le  h^2 \|u\|_{W^{4,\infty}(\hat{\Omega})} \|v\|_{L^1(\Omh \backslash \Omega)}
\end{alignat*}

However, we have 
\begin{alignat*}{1}
 \|v\|_{L^1(\Omh \backslash \Omega)} \le& C h^2 \|v\|_{L^\infty(\Omh \backslash \Omega)} \\
  \le & C h^3 \|\nabla v\|_{L^\infty(\Omega)} \\
  \le &C \, h^2 \|\nabla v\|_{L^2(\Omega)}= C \, h^2 \| v\|_{a}.
\end{alignat*}
Therefore, we get 
 \begin{equation*}
 \sup_{v \in \mW^k} \frac{|G_1(v)|}{|v|_a} \le C h^4  \|u\|_{W^{4,\infty}(\hat{\Omega})}.
\end{equation*}
\end{proof}

Now consider $v \in \Bh^k$.
\begin{alignat*}{1}
G_1(v)=&  h^4 \|u\|_{W^{4,\infty}(\hat{\Omega})}  \|v\|_{L^\infty(\Gamma)}+\|\delta^{-1/2}( \delta \frac{\partial u}{\partial n}+ u-\hat{g})\|_{L^\infty(\Gamma)}  \|v\|_{c}  \\
\le & h^{7/2}\|u\|_{W^{4,\infty}(\hat{\Omega})}  \|v\|_{L^2(\Gamma)}+ h^3 \|u\|_{W^{2,\infty}(\hat{\Omega})}  \|v\|_{c}   \\
\le & h^{9/2}\|u\|_{W^{4,\infty}(\hat{\Omega})}  |v|_c+ h^3 \|u\|_{W^{2,\infty}(\hat{\Omega})}  \|v\|_{c}.  
\end{alignat*}
Here we used  \eqref{taylor}.

Hence,
\begin{equation*}
 \sqrt{h} (\sup_{v \in \Bh^k} \frac{|G_1(v)|}{|v|_c}) \le C \,( h^{7/2}  \|u\|_{W^{2,\infty}(\hat{\Omega})}+ h^5 \|u\|_{W^{4,\infty}(\hat{\Omega})}). 
\end{equation*}

Now lets consider $G_2$. If we let $v \in \mW^k$ then

\begin{equation*}
G_2(v)=a_h(u-u_I, v) \le \|u-u_I\|_a \|v\|_a
\end{equation*}

Hence, 
\begin{equation*}
 \sup_{v \in \mW^k} \frac{|G_2(v)|}{\|v\|_a} \le C h^{k}\|u\|_{H^k(\hat{\Omega})}.
\end{equation*}
Now let $v \in \Bh^k$ we then have 
\begin{equation*}
G_2(v)= \|u-u_I\|_a \|v\|_a + |u-u_I|_c  |v|_c \le c_1 \sqrt{h} \|u-u_I\|_a \|v\|_c + |u-u_I|_c  |v|_c.
\end{equation*}

Let  $e \in \mathcal{E}_h$, $e \subset \Gamma$ with end points $x_0$ and $x_1$. Then, we have $|(u-u_I)(x)|^2 \le C |x-x_0||x-x_1| \|\partial_t (u-u_I)\|_{L^\infty(e)} $. Hence, using Assumption \ref{assumption1} we get
\begin{equation*}
\frac{(u-u_I)^2(x)}{|\delta(x)|} \le C \beta  \|\partial_t (u-u_I)\|_{L^\infty(e)}.
\end{equation*}
Thus,
\begin{equation*}
\int_e \frac{(u-u_I)^2}{|\delta|} ds \le  C \beta |e| \|\partial_t (u-u_I)\|_{L^\infty(e)}^2.
\end{equation*}
We then obtain the following estimate, after summing over all edges $e \subset \Gamma$,
 \begin{equation*}
 |u-u_I|_c^2 \le C \|\partial_t (u-u_I)\|_{L^\infty(\Gamma)}^2.
\end{equation*}
We get the following inequality after using approximation properties of the Lagrange interpolant:
\begin{equation*}
 |u-u_I|_c \le C h^{k} \|u\|_{W^{k+1,\infty}(\Gamma)}.
\end{equation*}

Therefore, we have 
\begin{equation*}
 \sqrt{h} \sup_{v \in \Bh^k} \frac{|G_2(v)|}{|v|_c} \le   C h^{k+1/2} (\|u\|_{W^{k+1,\infty}(\Gamma)}+ \|u\|_{H^{k+1}(\hat{\Omega})}).
\end{equation*}
Combining the above results we get

\begin{equation*}
\sup_{v \in \mW^k} \frac{|G(v_h)|}{\|v_h\|_a} \le  C \left(h^{k}\|u\|_{H^{k+1}(\hat{\Omega})} + h^4  \|u\|_{W^{4,\infty}(\hat{\Omega})}\right).
\end{equation*}

\begin{alignat*}{1}
 \sqrt{h} \sup_{v \in \Bh^k} \frac{|G(v)|}{|v|_c} \le & C \,\left( h^{7/2}  \|u\|_{W^{2,\infty}(\hat{\Omega})}+ h^5 \|u\|_{W^{4,\infty}(\hat{\Omega})}\right) \\
&+  C h^{k+1/2} \left(\|u\|_{W^{k+1,\infty}(\Gamma)}+ \|u\|_{H^{k+1}(\hat{\Omega})}\right).
\end{alignat*}

The result now follows from Theorem \ref{stability}.

\section{Implementation}

One feature of Nitsche's method, that is preserved with BDT, is that
one uses the full space $W^k_h$ of piecewise polynomials without restriction
at the boundary. The modification of $W^k_h$ to obtain the space $V_h^k$ of
piecewise polynomials vanishing at boundary vertices is not trivial to
implement in automated systems like FEniCS \cite{fenicsbook}.

Thus it is of interest to consider a simplification to the Robin-type
method \eqref{fem} which removes this constraint.
Thus we define, for $\epsilon>0$,
$$
b_h^\epsilon(u,v)= a_h(u,v)
   +c_h^\epsilon(u,v),
$$
where $c_h^\epsilon(u,v):= \int_{\Gamma}(\epsilon\,\hbox{sign}(\delta)+\delta)^{-1}{u}v \,ds$.
We then define $\hat{W}_h^k= \{v \in W_h^k: v=0 \text{ on }  \Gamma^0\}$ and  $\hat{W}_h^k(g)= \{v \in W_h^k: v=g_I \text{ on }  \Gamma^0\}$.

For implementation issues we solve $u_h \in \hat{W}^k_h(g)$ by
\begin{equation}\label{eqn:epsrobimet}
b_h^\epsilon(u_h,v_h)  =\int_{\Omega_h} Fv\,dx+ c_h^\epsilon(\hat{g},v)  \quad \forall\, v\in \hat{W}^k_h.
\end{equation}
The computational experiments used this approach.
The answers do not depend on $\epsilon$ for $\epsilon$ small,
as indicated in Table \ref{tabl:epsrobin}. We were even able to have $\epsilon=0$ for \eqref{eqn:epsrobimet} using {\tt dolfin}.

\begin{table}[h]
\begin{center}
\begin{tabular}{|c|c|c|c|c||c|c|c|}
\hline
$k$  &  $M$ & segs &   hmax  &$\epsilon$&  L2 err& H1 err& bdry err \\
\hline
  2  &  64  &  320 & 3.5e-02 & 1.0e-04  & 1.1e-03  & 2.1e-03  & 1.3e-01 \\
  2  &  64  &  320 & 3.5e-02 & 1.0e-05  & 1.1e-04  & 1.8e-03  & 2.5e-02 \\
  2  &  64  &  320 & 3.5e-02 & 1.0e-06  & 1.2e-05  & 1.8e-03  & 3.2e-03 \\
  2  &  64  &  320 & 3.5e-02 & 1.0e-07  & 6.0e-06  & 1.8e-03  & 3.2e-04 \\
  2  &  64  &  320 & 3.5e-02 & 1.0e-08  & 5.9e-06  & 1.8e-03  & 4.3e-05 \\
  2  &  64  &  320 & 3.5e-02 & 1.0e-09  & 5.9e-06  & 1.8e-03  & 3.1e-05 \\
  2  &  64  &  320 & 3.5e-02 & 1.0e-10  & 5.9e-06  & 1.8e-03  & 3.1e-05 \\
\hline
  2  &  128  &  640 & 1.8e-02 & 1.0e-07  & 1.3e-06  & 4.4e-04  & 6.4e-04 \\
  2  &  128  &  640 & 1.8e-02 & 1.0e-08  & 7.3e-07  & 4.4e-04  & 6.5e-05 \\
  2  &  128  &  640 & 1.8e-02 & 1.0e-09  & 7.2e-07  & 4.4e-04  & 7.3e-06 \\
  2  &  128  &  640 & 1.8e-02 & 1.0e-10  & 7.2e-07  & 4.4e-04  & 3.9e-06 \\
  2  &  128  &  640 & 1.8e-02 & 1.0e-11  & 7.2e-07  & 4.4e-04  & 3.9e-06 \\
\hline
  2  &  256  &  1280 & 9.0e-03 & 1.0e-09  & 8.9e-08  & 1.1e-04  & 1.3e-05 \\
  2  &  256  &  1280 & 9.0e-03 & 1.0e-10  & 8.9e-08  & 1.1e-04  & 1.3e-06 \\
  2  &  256  &  1280 & 9.0e-03 & 1.0e-11  & 8.9e-08  & 1.1e-04  & 4.9e-07 \\
  2  &  256  &  1280 & 9.0e-03 & 1.0e-12  & 8.9e-08  & 1.1e-04  & 4.9e-07 \\
\hline
\end{tabular}
\end{center}
\vspace{0mm}
\caption{Errors $\norm{u_h-u_I}_{L^2(\Omega_h)}$, 
$\norm{u_h-u_I}_{H^1(\Omega_h)}$, and
$\norm{\,|\delta|^{-1/2}(u_h-u_I)}_{L^2(\partial\Omega_h)}$ as a function of $\epsilon$ and
maximum mesh size (hmax) for the Robin-like approximation \eqref{fem}
but modified as in \eqref{eqn:epsrobimet}, for piecewise quadratic polynomials ($k=2$).
Key: $M$ is the value of the {\tt meshsize} input parameter to the {\tt mshr}
function {\tt circle} used to generate the mesh;
segs is the number of boundary edges.}
\label{tabl:epsrobin}
\end{table}

\section{Computational Experiments}
\subsection{Example of a circle}
We return now to the computational test problem described in Section \ref{circle}. 
It is not difficult to show that Assumption \ref{assumption1} holds for the meshes we used. 
We see from Table \ref{tabl:raterobin} that the $H^{1}(\Omega_h)$ error is
optimal order for $k\leq 3$, consistent with Theorem \ref{errorestimates}.
In these cases, the $L^{2}(\Omega_h)$ error is also optimal order,
and the boundary error is higher order for quadratics.  For $k \ge 4$ our numerical experiments seem to predict the error
$$
\|u-u_h\|_{H^1(\Omega_h)}\approx C\big(h^{7/2}+h^k\big),
$$
which coincides with Theorem \ref{errorestimates}. 

It appears from Table \ref{tabl:raterobin} that the boundary error term
$$
\norm{\,|\delta|^{-1/2}(u-u_h)}_{L^2(\partial\Omega_h)} \approx Ch^3, \quad \text{ for all } k \ge 2, 
$$
which is consistent with Theorem \ref{errorestimates}.

\begin{figure}
\centerline{(a)\includegraphics[width=3.0in]{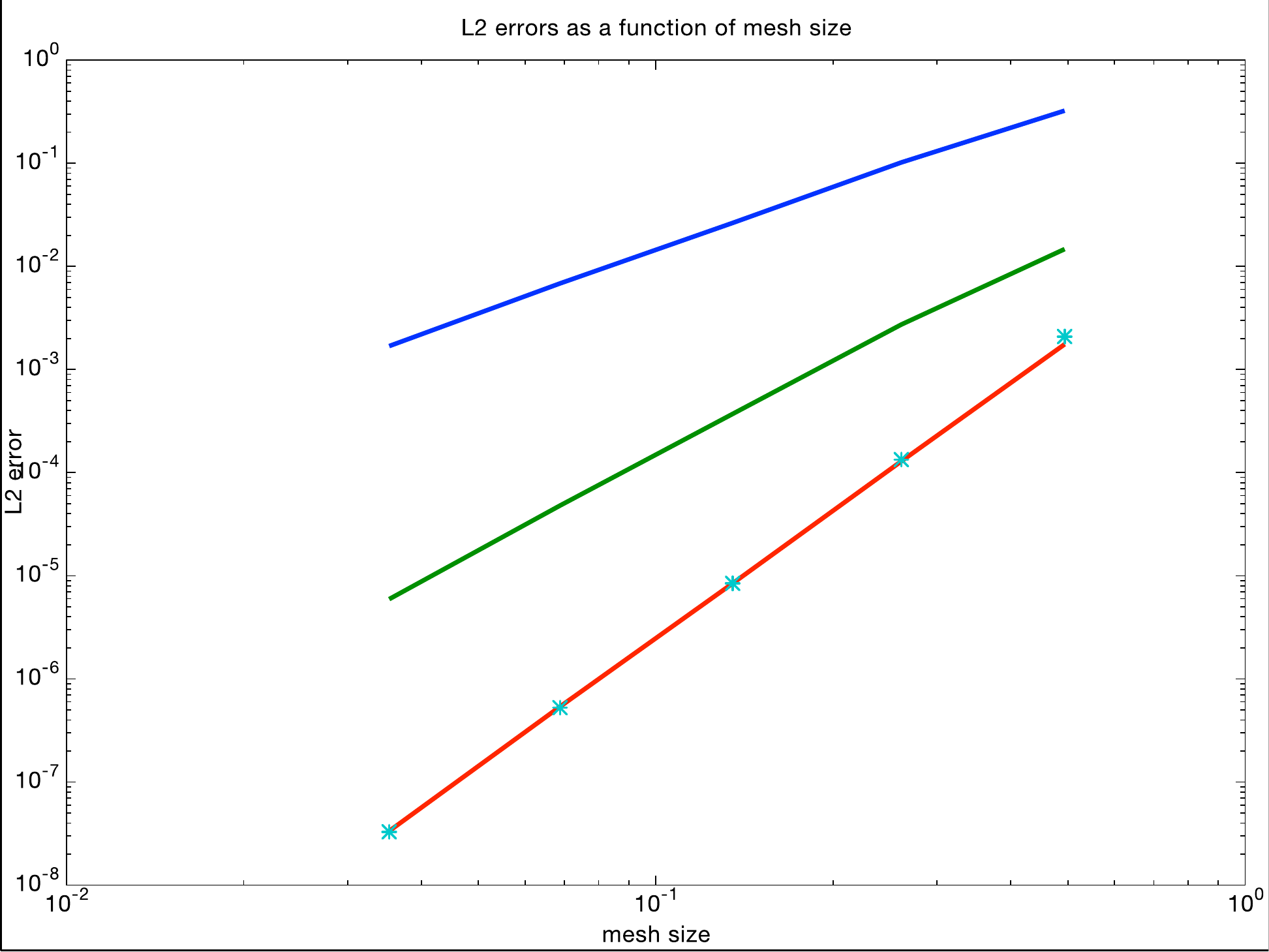}
(b)\includegraphics[width=3.0in]{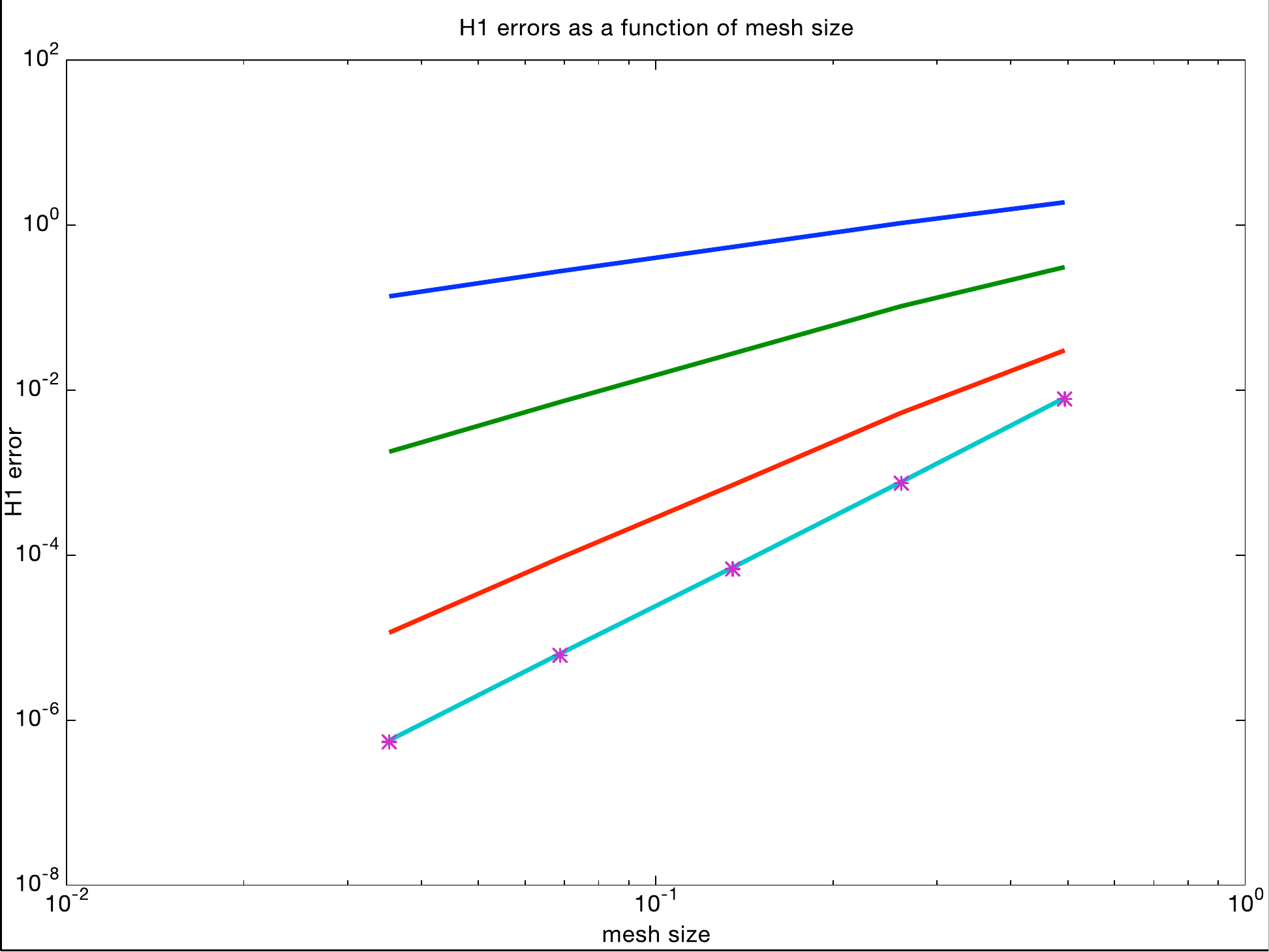}}
\caption{Errors $u_h-u_I$ in (a) $L^2(\Omega_h)$ and (b) $H^1(\Omega_h)$ 
as a function of the maximum mesh size for the  method \eqref{fem}.
The asterisks indicate data for (a) $k=4$ and (b) $k=5$.}
\label{fig:plotrobincor}
\end{figure}

\begin{table}[h]
\begin{center}
\begin{tabular}{|c|c|c||c|c||c|c||c|c|}
\hline
$k$&$M$ &hmax&  L2 error&rate& H1 error&rate&bdry err & rate\\
\hline
 1 &  16 &    0.135 &   0.0264 &   1.95 &    0.545 &  0.96 &  0.292 &  1.04  \\
 1 &  32 &   0.0688 &  0.00683 &   1.95 &    0.277 &  0.98&   0.145 & 1.01        \\
 1 &  64 &   0.0353 &  0.00169 &   2.01 &    0.137 &   1.02&  0.0724&  1.00 \\
\hline
 2 &  16 &    0.135 & 3.71e-04 &   2.88 &   0.0278 &   1.90 & 0.00177 & 2.71 \\
 2 &  32 &   0.0688 &  4.80e-05 &   2.95 &  0.00719 &  1.95& 2.52e-04 & 2.81 \\
 2 &  64 &   0.0353 & 5.94e-06 &   3.02 &  0.00179 &   2.00&  3.12e-05 & 3.02 \\
\hline
 3 &  16 &    0.135 & 8.43e-06 &   3.94 & 7.07e-04 &   2.91&  5.22e-04  &2.98 \\
 3 &  32 &   0.0688 & 5.39e-07 &   3.97 & 9.25e-05 &   2.93& 6.52e-05  &3.00   \\
 3 &  64 &   0.0353 & 3.35e-08 &   4.00 & 1.15e-05 &   3.01& 8.13e-06  &3.01 \\
\hline
 4 &  16 &    0.135 & 8.43e-06 &   3.99 & 7.07e-05 &   3.45& 5.34e-04 & 2.97 \\
 4 &  32 &   0.0688 & 5.27e-07 &   4.00 & 6.38e-06 &   3.47& 6.74e-05 &  2.99 \\
 4 &  64 &   0.0353 & 3.29e-08 &   4.00 & 5.69e-07 &   3.49&  8.47e-06 & 2.99 \\
\hline
 5 &  16 &    0.135 & 8.43e-06 &   3.99 &  6.80e-05 &   3.45 & 5.35e-04 &  2.97 \\
 5 &  32 &   0.0688 & 5.27e-07 &   4.00 & 6.11e-06 &   3.48  &6.75e-05  & 2.99 \\
 5 &  64 &   0.0353 &  3.30e-08 &   4.00 & 5.45e-07 &   3.49 & 8.47e-06  & 2.99\\
\hline
\end{tabular}
\end{center}
\vspace{0mm}
\caption{Errors $\norm{u_h-u_I}_{L^2(\Omega_h)}$, 
$\norm{u_h-u_I}_{H^1(\Omega_h)}$, and
$\norm{\,|\delta^{-1/2}(u_h-u_I)}_{L^2(\partial\Omega_h)}$ as a function of
mesh size (hmax) for the method  \eqref{eqn:epsrobimet}
for various polynomial degrees $k$. The fudge factor $\epsilon$ was taken to be $10^{-13}$.
Results were insignificantly different for smaller values.
Key: $M$ is the value of the {\tt meshsize} input parameter to the {\tt mshr}
function {\tt circle} used to generate the mesh.
The number of boundary edges was set to $5M$, and
hmax is the maximum mesh size.}
\label{tabl:raterobin}
\end{table}

\subsection{An example with $\delta<0$}
\label{sec:testprobis}

Now consider the case where $\Omega$ is a disc of radius $1$ centered at 
the origin, having a concentric disc of radius $R<1$ removed. Again, it is not difficult to show that Assumption \ref{assumption1} holds for our meshes.

For boundary value problem, we take $R=\half$ and $-\Delta u=f$, with
$$
u(x,y)=(x^2+y^2) -5(x^2+y^2)^2+4(x^2+y^2)^3, 
\qquad f=-4 +80 (x^2+y^2) -144 (x^2+y^2)^2
$$
in the computational experiments described in Table \ref{tabl:rateshortrobin}.
Note that $u$ vanishes on both boundary arcs. Note that the error estimates are consistent with Theorem \ref{errorestimates}.

\begin{table}[h]
\begin{center}
\begin{tabular}{|c|c|c|c|c|c|}\hline
$k$  &  $M$ &  hmax & L2 error & H1 error& bdry error\\
\hline
  2  &  16  & 0.132 & 8.76e-04  & 6.87e-02  & 1.39e-04 \\  
  2  &  32  & 0.070 & 1.20e-04  & 1.84e-02  & 9.64e-06 \\  
  2  &  64  & 0.036  & 1.54e-05  & 4.68e-03  & 6.51e-07 \\  
\hline
  3  &  16  & 0.132 & 2.90e-05  & 2.29e-03  & 6.59e-05 \\  
  3  &  32  & 0.070 & 1.89e-06  & 3.07e-04  & 4.13e-06 \\  
  3  &  64  & 0.036  & 1.17e-07  & 3.93e-05  & 2.47e-07 \\  
\hline
  4  &  16  & 0.132  & 2.23e-05 & 3.37e-04  & 7.24e-05 \\  
  4  &  32  & 0.070 & 1.39e-06 & 2.97e-05  & 4.57e-06 \\  
  4  &  64  & 0.036 & 8.10e-08 & 2.61e-06  & 2.76e-07 \\  
\hline
\end{tabular}
\end{center}
\vspace{0mm}
\caption{Errors $u_h-u_I$ measured in $L^2(\Omega_h)$ (L2 error), 
$H^1(\Omega_h)$ (H1 error), and $L^2(\partial\Omega_h)$ (bdry error)
as a function of mesh size (hmax) for the the Robin approximation 
in \eqref{eqn:epsrobimet}, for selected polynomial degrees $k$.
$\epsilon=10^{-9}$.
Key: $M$ is the value of the {\tt meshsize} input parameter to the {\tt mshr}
function {\tt circle} used to generate the mesh.
The number of boundary edges for the outer boundary was set to $4M$, and
the number of boundary edges for the inner boundary was set to $2M$.}
\label{tabl:rateshortrobin}
\end{table}

\section{Boundary layers}

It is natural to expect the error with various boundary approximations might be 
limited to a boundary layer, with the interior error of a smaller magnitude.
Our observations indicate something like this, but the behavior is more complex.
In Figure \ref{fig:bdrywhat}, we see two computations done on the same mesh
based on a triangulation of $\Omega_h$ with $\partial\Omega_h$ having 80 segments
and using piecewise-quadratic approximation.
In Figure \ref{fig:bdrywhat}(a), we see the simple polygonal approximation \eqref{eqn:polyapprocirkl}.
In this case, the error is somewhat larger near the boundary, but it does
not decay to zero in the interior.
Thus there is a significant pollution effect away from the boundary.
On the other hand, Figure \ref{fig:bdrywhat}(b) shows what happens if the
Robin-like method \eqref{fem}.
Now we see that the error does decay towards zero in the interior, with the
majority of the error concentrated at the boundary.

\begin{figure}
\centerline{(a)\includegraphics[width=3.0in]{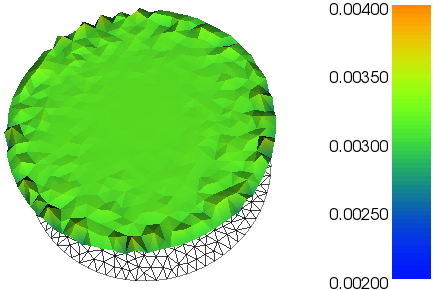} \quad    (b)
\hspace{-9pt} \includegraphics[width=3.2in]{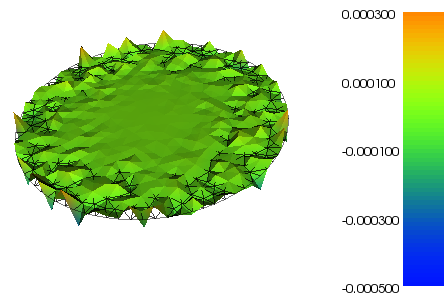}}
\caption{Error with piecewise quadratics 
on a mesh with $\partial\Omega_h$ having 80 segments.
The mesh is drawn in the plane corresponding to zero error.
(a) The method \eqref{eqn:polyapprocirkl}, no boundary integral corrections.
The error is uniformly positive.
(b) The Robin-like method \eqref{fem}.
The error oscillates around zero.
Note the factor of ten difference in scales in the error plots.}
\label{fig:bdrywhat}
\end{figure}

\section{Higher order and symmetric methods}
\label{sec:hoasm}

The Robin-type method presented in the previous section is at most of $O(h^{7/2})$. 
High-order methods using the same techinique do not lead to symmetric systems. 
For simplicity assume that $g \equiv 0$. Using that
$$
\Big|u|_{\partial\Omega_h}+\delta\derdir{u}{n}\big|_{\partial\Omega_h}
+\frac{\delta^2}{2}\derdirtwo{u}{n}\big|_{\partial\Omega_h}\Big|
       \leq C\delta^3\norm{u}_{W^3_\infty(\Omega)},
$$
we define
\begin{equation}\label{eqn:presecndordr}
b_h(u,v)= a_h(u,v)+\int_{\partial\Omega_h}\delta^{-1}{u}v \,ds
+\int_{\partial\Omega_h} \frac{\delta}{2}\derdirtwo{u}{n}v \,ds .
\end{equation}
Unfortunately, $b_h$ is not symmetric.

One way to have higher-order, symmetric methods is by symmetrizing 
the approach of Bramble-Dupont-Thom\'ee. 
Recall that Bramble et al. \cite{bramble1972projection} developed arbitrary order methods,
but that the bilinear forms are not symmetric.  
The lowest order method was presented in Section \ref{sec:BDTmeth} where the bilinear $N_h$ 
is given by \eqref{eqn:toddformr}.    
One way to symmetrize $N_h$ and mainting the same convergence rates is by introducing the 
bilinear form:
\begin{equation*}
M_h(u,v)=N_h(u,v) 
+ \int_{\partial\Omega_h} \gamma \delta h^{-1} \derdir{v}{n} \Big(u+\delta \derdir{u}{n}\Big) \,ds.
\end{equation*}
This is precisely what is done in \cite[(2.31)]{ref:CutFEMbasedonBDT}.
We see that
\begin{alignat*}{1}
M_h(u,v)=&a_h(u,v)+\int_{\partial\Omega_h} \Big(\gamma \frac{\delta}{h}-1\Big)
 \Big( \delta  \derdir{u}{n} \derdir{v}{n} + \derdir{u}{n} v+\derdir{v}{n} u\Big) \,ds
 + \frac{\gamma}{h} \int_{\partial\Omega_h} u v \,ds.
\end{alignat*}
Note that $M_h$ is symmetric. We will investigate this and similar methods in the near future.

\end{document}